\documentclass[a4paper, 11pt]{amsart}
\usepackage{mathptmx,amssymb,amscd,latexsym, eulervm}
\usepackage{amsmath}
\usepackage{amsthm}
\usepackage{mathdots}
\usepackage[colorlinks=true,citecolor=violet,linkcolor=blue,urlcolor=blue]{hyperref}
\usepackage[dvipsnames]{xcolor}
\usepackage[onehalfspacing]{setspace}
\usepackage{tabularx}
\usepackage{amsfonts}
\usepackage{paralist}
\usepackage{aliascnt}
\usepackage{amscd}
\usepackage{blkarray}
\usepackage{mathbbol}
\usepackage{setspace}
\usepackage{needspace}
\usepackage[inner=2.4cm,outer=2.4cm, bottom=3.2cm]{geometry}
\usepackage{tikz, tikz-cd}
\usepackage{calligra,mathrsfs}

\usepackage{tikz}
\usetikzlibrary{matrix}
\usetikzlibrary{arrows,calc}
\allowdisplaybreaks

\newtheorem{theorem}{Theorem}[section]
\newtheorem{headthm}{Theorem}

\newaliascnt{proposition}{theorem}
\newtheorem{proposition}[proposition]{Proposition}
\aliascntresetthe{proposition}

\newaliascnt{lemma}{theorem}
\newtheorem{lemma}[lemma]{Lemma}
\aliascntresetthe{lemma}

\newaliascnt{corollary}{theorem}

\aliascntresetthe{corollary}

\theoremstyle{definition}
\newaliascnt{construction}{theorem}

\aliascntresetthe{construction}

\newaliascnt{definition}{theorem}
\newtheorem{definition}[definition]{Definition}
\aliascntresetthe{definition}

\newaliascnt{example}{theorem}

\aliascntresetthe{example}

\newaliascnt{remark}{theorem}
\newtheorem{remark}[remark]{Remark}
\aliascntresetthe{remark}

\newaliascnt{question}{theorem}

\aliascntresetthe{question}

\def\equationautorefname~#1\null{(#1)\null}
\def\sectionautorefname~#1\null{Section #1\null}
\def\subsectionautorefname~#1\null{\S #1\null}

\newcommand{\NN}{\mathbb{N}}
\newcommand{\QQ}{\mathbb{Q}}
\newcommand{\ZZ}{\mathbb{Z}}

\newcommand{\PP}{\mathbb{P}}

\newcommand{\HH}{\normalfont\text{H}}
\newcommand{\kk}{\mathbb{k}}
\newcommand{\Supp}{\operatorname{Supp}}
\newcommand{\multProj}{\operatorname{MultiProj}}
\newcommand{\Newt}{\operatorname{Newt}}
\newcommand{\conv}{\operatorname{Conv}}
\newcommand{\rank}{\operatorname{rank}}
\DeclareMathOperator{\rk}{rk}
\newcommand{\codim}{\operatorname{codim}}
\newcommand{\K}{\mathcal{K}}
\newcommand{\LL}{\mathscr{L}}
\newcommand{\EE}{\mathscr{E}}
\newcommand{\OO}{\mathscr{O}}
\newcommand{\G}{\mathfrak{G}}

\newcommand{\zero}{\mathbf{0}}
\newcommand{\one}{\mathbf{1}}
\newcommand{\ba}{\mathbf{a}}
\newcommand{\bc}{\mathbf{c}}
\newcommand{\bd}{\mathbf{d}}
\newcommand{\br}{\mathbf{r}}
\newcommand{\bs}{\mathbf{s}}
\newcommand{\bh}{\mathbf{h}}

\newcommand{\bx}{\mathbf{x}}
\newcommand{\by}{\mathbf{y}}

\newcommand{\bt}{\mathbf{t}}
\newcommand{\bn}{\mathbf{n}}
\newcommand{\bm}{\mathbf{m}}
\newcommand{\bb}{\mathbf{b}}

\newcommand{\ee}{{\normalfont\mathbf{e}}}
\DeclareMathOperator{\MultiProj}{MultiProj}
\DeclareMathOperator{\Spec}{Spec}

\DeclareMathOperator{\Sym}{Sym}
\DeclareMathOperator{\hsupp}{HSupp}
\DeclareMathOperator{\msupp}{MSupp}
\DeclareMathOperator{\dmsupp}{DMSupp}
\DeclareMathOperator{\Ker}{Ker}

\DeclareFontFamily{OT1}{pzc}{}
\DeclareFontShape{OT1}{pzc}{m}{it}{<-> s * [1.100] pzcmi7t}{}
\DeclareMathAlphabet{\mathchanc}{OT1}{pzc}{m}{it}

\DeclareMathOperator{\fProj}{\mathchanc{Proj}}

\title{On the support of double Grothendieck polynomials}
\author{Yairon Cid-Ruiz}
\address{Department of Mathematics, North Carolina State University, Raleigh, NC 27695, USA}
\email{ycidrui@ncsu.edu}
\date{September 13, 2026}
\subjclass[2020]{14M15, 14C17, 13D40, 05E05, 52B40}
\keywords{Double Grothendieck polynomials, matrix Schubert varieties, Hilbert coefficients, rational singularities, standardization, $M$-convex sets, $M^\natural$-convex sets, saturated Newton polytopes}

\begin{document}
\begin{abstract}
We prove that the support of every double Grothendieck polynomial is an $M^\natural$-convex set.
Our main new tool is a rigidity result for the $K$-classes of multiprojective varieties with rational singularities. 
\end{abstract}

\maketitle

\tableofcontents

\section{Introduction}

Grothendieck polynomials are important objects in algebraic combinatorics and algebraic geometry.
They were introduced by Lascoux and Sch\"utzenberger \cite{LS} to represent the classes of structure sheaves of Schubert varieties in the $K$-theory of flag varieties.
These polynomials have been extensively studied (see, e.g., \cite{FK,Lenart,Buch,KM,LRS, Weigandt,PSW} and the references therein).
In this paper, we study the supports of double Grothendieck polynomials.

We first recall the definition of Grothendieck polynomials.
Consider the polynomial ring $\ZZ[\bx, \by] = \ZZ[x_1,\ldots,x_n,y_1,\ldots,y_n]$.
For a permutation $w\in S_n$, denote by $\ell(w)$ the number of inversions.
Let $s_i=(i,i+1)$ be the $i$-th simple transposition and $w_0=(n,n-1,\ldots,1)$ the longest permutation.

\begin{definition}[Grothendieck polynomials]\label{def:grothendieck}
Let $s_i$ interchange $x_i,x_{i+1}$ and fix $\by$, and set
$$
\partial_i(f) \;=\; \frac{f-s_i(f)}{x_i-x_{i+1}} \qquad \text{ and } \qquad
\pi_i(f) \;=\; \partial_i\bigl((1-x_{i+1})f\bigr)
$$
for any $f \in \ZZ[\bx, \by]$.
The \emph{double Grothendieck} polynomials are determined by
$$
\begin{aligned}
\G_{w_0}(\bx;\by)&\;=\; \prod_{i+j\le n}(x_i+y_j-x_iy_j),\\
\G_{ws_i}(\bx;\by)&\;=\; \pi_i\bigl(\G_w(\bx;\by)\bigr)
\quad\text{if \;}\ell(ws_i)\;=\; \ell(w)-1.
\end{aligned}
$$
The \emph{ordinary Grothendieck polynomial} is $\G_w(\bx)=\G_w(\bx;\zero) \in \ZZ[\bx]$.
\end{definition}

For a polynomial $f$ in $N$ variables, let $\Supp(f)\subset\NN^N$ be its exponent support and $\Newt(f):=\conv(\Supp(f))$ be its Newton polytope.
Following \cite{MTY}, we say that $f$ has the \emph{saturated Newton polytope property} (SNP) if every lattice point of $\Newt(f)$ belongs to $\Supp(f)$.
Monical, Tokcan, and Yong \cite{MTY} conjectured this property for many polynomials in algebraic combinatorics.
Several of their conjectures have already been settled positively (see \cite{EY,FMS,MSD,HMMSD,CCMM,HMSSD,KMF,NW}).

\smallskip

In this paper, we study the support of double Grothendieck polynomials, and our main result is the following:

\begin{headthm}\label{thm:main}
For every $w\in S_n$, the support of $\G_w(\bx;\by)$ is an $M^\natural$-convex set.
In particular, double Grothendieck polynomials have the SNP property.
\end{headthm}

Our approach uses the standardization construction of \cite{CCMM,CCC,CLM}.
By Knutson--Miller's formula \cite{KM}, double Grothendieck polynomials are the twisted $K$-polynomials of matrix Schubert varieties.
We use a version of standardization that preserves the full $K$-polynomial and rational singularities.
The main new ingredient is a rigidity theorem for the signed Hilbert coefficients of a multiprojective variety with rational singularities, refining Brion's positivity theorem \cite{Brion} (see \autoref{thm:rigidity}).
Standardization allows us to apply this theorem to the supports of double Grothendieck polynomials.
We then combine the resulting support properties with Nguyen-Dang and Wang's result \cite{NW} that ordinary Grothendieck polynomials have $M^\natural$-convex support and with the top-degree factorization of Pechenik, Speyer, and Weigandt \cite{PSW}.

\medskip

\noindent
\textbf{Convention.}
Throughout this paper, unless otherwise stated, all algebraic varieties are defined over an algebraically closed field $\kk$ of characteristic zero. 

\medskip

\noindent
\textbf{Outline.}
The structure of the paper is as follows.
In \autoref{sec:poly}, we recall basic properties of $M^\natural$-convex sets. 
In \autoref{sec:rigidity}, we prove the rigidity theorem for signed Hilbert coefficients.
In \autoref{sec:standardization}, we develop the standardization construction and apply it to double Grothendieck polynomials.
In \autoref{sec:proof}, we prove the main theorem.
Finally, in \autoref{sec:ordinary}, we give a short, self-contained geometric proof that the positive homogenizations of ordinary Grothendieck polynomials are denormalized Lorentzian, a fact needed in our proof of \autoref{thm:main}.

\section{\texorpdfstring{$M^\natural$}{M-natural}-convex sets}\label{sec:poly}
We recall the properties of $M^\natural$-convex sets used in the proof of the main theorem. 
For more details on these topics, see \cite{Murota}, \cite{Schrijver}, \cite{HH}.

Let $n\geq1$ be a positive integer, and write $[n]:=\{1,\ldots,n\}$.
For two multiindices $\ba=(a_1,\ldots,a_n),\bb=(b_1,\ldots,b_n)\in\ZZ^n$,
we write $\ba\geq\bb$ whenever $a_i\geq b_i$ for all $i\in[n]$,
and define the weight of $\ba$ as $|\ba|:=a_1+\cdots+a_n$.
For each $i\in[n]$, let $\ee_i\in\NN^n$ be the $i$-th standard basis vector,
whose $i$-th entry is $1$ and all other entries are $0$.
We denote by $\zero=(0,\ldots,0)$ and $\one=(1,\ldots,1)$ the zero and
all-ones vectors, respectively.
For variables $\bt=(t_1,\ldots,t_n)$, we write
$\bt^{\ba}:=t_1^{a_1}\cdots t_n^{a_n}$.

\begin{definition}\label{def:discrete-exchange}
A finite nonempty set $B\subseteq\NN^n$ is \emph{$M$-convex} if it satisfies \emph{symmetric exchange}:
for all $\ba,\bb\in B$ and every $i$ with $a_i>b_i$, there exists $j$
with $a_j<b_j$ such that
$$
 \ba-\ee_i+\ee_j \;\in\; B
 \qquad\text{and}\qquad
 \bb+\ee_i-\ee_j \;\in\; B.
$$
Equivalently, we say that $B$ is a \emph{polymatroid}.
A finite nonempty set $S\subseteq\NN^n$ is \emph{$M^\natural$-convex} if
its homogenization
$$
 \widehat S_D \;:=\; \big\{(\ba,D-|\ba|)\mid\ba\in S\big\} \;\subseteq\; \NN^{n+1}
$$
is $M$-convex, where $D\ge\max\big\lbrace |\ba| \,\mid\, {\ba\in S} \big\rbrace$ is an integer.
Equivalently, we say that $S$ is a \emph{generalized polymatroid}.
\end{definition}

\begin{remark}[{\cite[\S5]{KMF}}]\label{prop:exchange-criterion}
To show that a finite nonempty set $S\subseteq\NN^n$ is $M^\natural$-convex, it is enough to verify the following two conditions.
\begin{enumerate}[\rm (i)]
\item \emph{Exchange:} For every pair $(\ba,\bb)\in S\times S$ and
index $i\in\{1,\ldots,n\}$ such that $a_i>b_i$, at least one of the
following conditions holds:
\begin{enumerate}[\rm (1)]
	\item There exists $j\in\{1,\ldots,n\}$ with $a_j<b_j$ such that $\ba-\ee_i+\ee_j\in S$ and $\bb-\ee_j+\ee_i\in S$.
	\item $|\ba|>|\bb|$,  and $\ba-\ee_i\in S$ and $\bb+\ee_i\in S$.
\end{enumerate}
\item \emph{Expansion:} For every pair $(\ba,\bb)\in S\times S$ with
$|\ba|<|\bb|$, there exists an index $j\in\{1,\ldots,n\}$ such that $a_j<b_j$,  and $\ba+\ee_j\in S$ and $\bb -\ee_j \in S$.
\end{enumerate}
\end{remark}

\begin{remark}\label{rem_M_natural_top_bot}
	Every nonempty fixed-degree slice of an $M^\natural$-convex set is $M$-convex. 
	Indeed, two points in such a slice have equal homogenizing coordinates, so symmetric exchange uses only the original coordinates and preserves their total degree.
	Moreover, every point is coordinatewise bounded below by a point of
	minimum total degree and above by a point of maximum total degree.
	Indeed, fix points of minimum and maximum total degree and repeatedly
	apply \autoref{prop:exchange-criterion}(ii) to decrease or increase the
	given point until the respective degree is reached.
\end{remark}

For an $M$-convex set $C\subseteq\NN^n$, the corresponding rank function is given by
$$
\rk_C \;:\; 2^{[n]} \rightarrow \NN, \qquad
 \rk_C(I)\;:=\;\max\Bigg\lbrace \sum_{i\in I}c_i \;\mid\; \bc = (c_1,\ldots, c_n)\in C \Bigg\rbrace.
$$
We will use the following rank criterion.

\begin{lemma}\label{rem:intermediate-support}
Let $B,T\subseteq\NN^n$ be $M$-convex sets.
If the function $\rk_T-\rk_B$ is nondecreasing under inclusion, then the set
$$
 S\;:=\;\big\{\ba\in\NN^n\mid\bb\le\ba\le\bt
       \text{ for some }\bb\in B,\ \bt\in T\big\}
$$
is $M^\natural$-convex.
Conversely, if $B$ and $T$ are the subsets of minimum and maximum total degree of an $M^\natural$-convex set, respectively, then $\rk_T-\rk_B$ is nondecreasing under inclusion.
\end{lemma}
\begin{proof}
$(\Rightarrow)$ Assume that $\rk_T-\rk_B : 2^{[n]} \rightarrow \NN$ is nondecreasing under inclusion.
Let  $d:=\rk_B([n])$, and $D:=\rk_T([n])$.
Notice that $D - d \ge \rk_T(I) - \rk_B(I) \ge \rk_T(\varnothing) - \rk_B(\varnothing) = 0$ for all $I \subseteq [n]$.
We consider the function $h : 2^{[n+1]} \rightarrow \NN$ given by 
$$
 h(I) \;:= \;\begin{cases}
	\rk_T\left(I\right) & \quad \text{if\; } I\subseteq [n] \\
	D-d+\rk_B\left(I\setminus\{n+1\}\right) & \quad \text{if\; } n+1 \in I.
 \end{cases}
$$
We show that $h$ is a polymatroid rank function.
Clearly, $h$ is integer-valued and $h(\varnothing)=0$.
Monotonicity follows from that of $\rk_B$ and $\rk_T$, together with
$$
 h\left(I\cup\{n+1\}\right)-h\left(I\right)
 \;=\;(D-d)-\bigl(\rk_T(I)-\rk_B(I)\bigr)\;\geq\;0 \qquad \text{ for all } I \subseteq [n].
$$
To check submodularity, only pairs with exactly one set containing $n+1$ need to be considered. 
For $I,J\subseteq[n]$, the hypothesis and submodularity of $\rk_T$ yield
\begin{align*}
	h\left(\left(I \cup \{n+1\}\right) \cup J\right) - h\left(I \cup \{n+1\}\right) &\;=\; \rk_B(I\cup J)-\rk_B(I)\\
	&\;\le\; \rk_T(I\cup J)-\rk_T(I)\\
	&\;\leq\;\rk_T(J)-\rk_T(I\cap J)\\
	&\;\leq\; h(J)-h\left(\left(I \cup \{n+1\}\right) \cap J\right).
\end{align*}
Therefore $h$ is a rank function with $h([n+1])=D$.
Then we have the equality
$$
 \widehat S_D
 \;=\;\Big\{\bc\in\NN^{n+1}\mid |\bc|=D \quad \text{and}\quad
       \sum_{i\in I}c_i\leq h(I)\text{ for all }I\subseteq[n+1]\Big\}.
$$
Indeed, for $\bc=(\ba,D-|\ba|)$, these inequalities are equivalent to
$$
 d-\rk_B([n]\setminus I)
 \;\leq\;\sum_{i\in I}a_i\;\leq\;\rk_T(I)
 \qquad\text{for all }I\subseteq[n].
$$
By the rank descriptions of the upward closure of $B$ and the downward closure of $T$, the lower and upper bounds say precisely that $\bb\leq\ba\leq\bt$ for some $\bb\in B$ and $\bt\in T$.
Finally $\widehat S_D$ is the $M$-convex set with rank function $h$, and so $S$ is $M^\natural$-convex.

$(\Leftarrow)$ Let $S$ be an $M^\natural$-convex set with minimum- and maximum-degree subsets $B$ and $T$, respectively.
Let $d:=\rk_B([n])$ and $D:=\rk_T([n])$, and  $h$ be the rank function of $\widehat S_D$.
By \autoref{rem_M_natural_top_bot}, for all $I \subseteq [n]$, we obtain
$$
 h(I) \;=\;\rk_T(I) \qquad \text{and} \qquad
 h(I\cup\{n+1\}) \;=\; D-d+\rk_B(I).
$$
By submodularity, $h(I\cup\{n+1\})-h(I)$ is nonincreasing under inclusion, and so $\rk_T-\rk_B$ is nondecreasing under inclusion.
\end{proof}

\section{Rigidity of signed Hilbert coefficients}\label{sec:rigidity}

In this section, we study the coefficients of the multigraded Hilbert polynomial of a multiprojective variety with rational singularities.
Equivalently, we study the $K$-class of a multiprojective variety inside the $K$-ring of coherent sheaves on the corresponding multiprojective space.
Our result extends previous work of Brion \cite{Brion} by providing a new rigidity result. 

Consider the product of projective spaces $\PP := \PP_\kk^{m_1}  \times_\kk \cdots \times_\kk \PP_\kk^{m_n} = \multProj(S)$ with multihomogeneous coordinate ring 
$$
S \;:=\; \kk\left[x_{i,j} \mid 1 \le i \le n, \, 0 \le j \le m_i\right] \qquad \text{ where } \deg(x_{i,j}) = \ee_i.
$$ 
Let $X \subset \PP$ be an integral subvariety with rational singularities. 
For the definition and basic properties of rational singularities, see \cite{Elkik,KollaMori}.
Let $d :=\dim(X)$.
The multigraded Hilbert polynomial (see \cite{KT, CR_MIXED_MULT}) is given by
 $$
 P_X(\bt) \;:=\; \chi\left(X,\OO_X(\bt)\right) \;=\; \sum_{i=0}^d {(-1)}^i h^i\left(X, \OO_X(t_1,\ldots,t_n)\right) \;\in\; \QQ[t_1, \ldots, t_n]. 
 $$ 
 We can write this polynomial in the binomial basis
$$
P_X(\bt) \;=\;\sum_{\bn\in\NN^n}e_{\bn}(X)
\prod_{i=1}^n\binom{t_i+n_i}{n_i} \qquad \text{ with } \qquad e_\bn(X) \;\in \; \ZZ.
$$
Then we set $c_{\bn}(X):=(-1)^{d-|\bn|}e_{\bn}(X)$.
The next definition introduces the objects of main interest in this section. 

\begin{definition}
	The \emph{Hilbert support of $X$} is given by the set 
	$$
	\hsupp_\PP(X) \;:=\; \big\{\bn \in \NN^n  \mid e_\bn(X) \neq 0 \big\}.
	$$
	The \emph{multidegree support of $X$} is given by the set 
	$$
	\msupp_\PP(X) \;:=\; \big\{\bn \in \NN^n  \mid |\bn| = d \text{  and } e_\bn(X) \neq 0 \big\}.
	$$
	We also consider the \emph{downward support of $X$} determined by multidegrees
	$$
	\dmsupp_\PP(X) \;:=\; \big\{\bn \in \NN^n  \mid \bn \le \bm \;\text{ for some }\; \bm \in \msupp_\PP(X) \big\}.
	$$
\end{definition}

The main result of this section is the following theorem. 

\begin{theorem}\label{thm:rigidity}
Let $X \subset \PP = \PP_\kk^{m_1} \times_\kk \cdots \times_\kk \PP_\kk^{m_n}$ be a $d$-dimensional integral subvariety with rational singularities. 
Then the following statements hold:
\begin{enumerate}[\rm (i)]
	\item $c_\bn(X) \ge 0$ for all $\bn \in \NN^n$.
	\item $\hsupp_\PP(X) \subseteq \dmsupp_\PP(X)$.
	\item Let $\bn, \bm \in \NN^n$ such that $\bn \le \bm$ and $\bm \in \dmsupp_\PP(X)$.
	Then 
	$$
	c_\bn(X) \;>\;0 \qquad \text{ implies that } \qquad c_\bm(X) \;>\; 0.
	$$
\end{enumerate}
\end{theorem}
The nonnegativity is Brion's theorem \cite{Brion}, in the form presented in \cite[Theorem 3.6]{KMF}. We follow that proof to obtain the additional strictness.

\begin{proof}
	To avoid non-irreducible possible scenarios, we use the elegant treatment of generic Bertini theorems as developed in \cite[\S 1.5]{FOV}.
	In particular, we replace $\kk$  by the purely transcendental field extension $\kk\left(z_{i,j} \mid 1 \le i \le n,\, 0 \le j \le m_i\right)$.
	Let $H_i = V\left(\sum_{j=0}^{m_i} z_{i,j}x_{i,j}\right) \subset \PP$ be the pullback of the generic hyperplane in the $i$-th factor $\PP_\kk^{m_i}$ of $\PP$.
	Then $X \cap H_i$ is either empty or integral of dimension $d-1$, and we have the equality 
	$$
	e_{\bn}(X) \;=\; e_{\bn-\ee_i}\left(X \cap H_i\right) \qquad \text{ for all \qquad $\bn \ge \ee_i$}.
	$$
	Thus, by successively cutting with generic hyperplanes, we can check the claims in parts (i), (ii), (iii) with the simplification that $\bn = \mathbf{0}$.
	Additionally, part (iii) can be checked with $\bn = \mathbf{0}$ and $\bm = \ee_i$ for some $i=1,\ldots,n$. 
	By using the adaptation of Bertini's theorem for rational singularities \cite[Remark 3.4.11(3)]{FOV}, we obtain that each $X \cap H_i$ also has rational singularities.

	Now part (ii) is clear. 
	Indeed, if $\mathbf{0} \not\in \dmsupp_\PP(X)$, then $\msupp_\PP(X) = \varnothing$, and therefore we should have $X = \varnothing$.
	
	Let $f : \widetilde{X} \rightarrow X$ be a rational resolution of $X$.
	Consider the ample line bundle $\LL := \OO_X(1, \ldots,1)$.
	The pullback line bundle $f^*(\LL)$ is big and nef on $\widetilde{X}$.
	By utilizing the Leray spectral sequence and the Kawamata--Viehweg vanishing theorem (see \cite[\S 5]{EV}), we obtain 
	$$
	\HH^i\Big(X, \LL^{-1}\Big) \;\cong\; \HH^i\Big(\widetilde{X}, f^*(\LL)^{-1}\Big) \;=\; 0
	$$ 
	for all $i < \dim(X)$.
	Therefore, we obtain 
	$$
	c_{\mathbf{0}}(X) \;=\; {(-1)}^{d}P_X\left(-1,\ldots,-1\right) \;=\; \dim_\kk\left(\HH^d\left(X, \LL^{-1} \right)\right) \;\ge\; 0.
	$$
	This shows the claimed nonnegativity statement in part (i).
	
	We assume that $c_{\mathbf{0}}(X) > 0$ and that $\ee_i \in \dmsupp_\PP(X)$. 
	Hence, to conclude the proof, we must show that $c_{\ee_i}(X) > 0$.
	Since $X$ is Cohen--Macaulay, by Serre duality (see \cite[\S III.7]{HARTSHORNE}), we obtain the equalities 
	$$
	c_\mathbf{0}(X) \;=\; h^0\left(X, \omega_X \otimes \LL\right)
	$$
	and 
	$$
	c_{\ee_i}(X) \;=\; h^0\left(X, \omega_X \otimes \LL \otimes \OO_X(\ee_i)\right) - h^0\left(X, \omega_X \otimes \LL\right) ;
	$$
	similarly, since $\LL \otimes \OO_X(\ee_i)$ is ample, we have the vanishing $\HH^j\left(X,\omega_X \otimes \LL \otimes \OO_X(\ee_i) \right)=0$ for all $j \ge 1$.

	Consider the finite dimensional $\kk$-vector spaces 
	$$
	W \;:=\; \HH^0\left(X, \omega_X \otimes \LL\right) \qquad \text{ and } \qquad V \;:=\; \HH^0\left(X, \omega_X \otimes \LL \otimes \OO_X(\ee_i)\right).
	$$
	Let $\Pi_i : \PP \rightarrow \PP_\kk^{m_i}$ be the natural projection to the $i$-th factor.
	Since there is some $\bd = (d_1,\ldots,d_n) \in \msupp_\PP(X)$ with  $d_i > 0$, \cite[Theorem A]{CCLMZ} implies that $\dim\left(\Pi_i(X)\right) > 0$.
	Hence we get global sections $s_0, s_1 \in \HH^0(X, \OO_X(\ee_i))$ such that the rational function
	$$
	r \;:=\; \frac{s_1}{s_0} \;\in\; K(X) \quad \text{ is transcendental over $\kk$.}
	$$
	Since $\omega_X \otimes \LL$ is torsion-free of rank one, we have an injection $\omega_X \otimes \LL \hookrightarrow \omega_X \otimes \LL \otimes \mathscr{K}_X \cong \mathscr{K}_X$, where $\mathscr{K}_X$ is the constant sheaf on $X$ with values in $K(X)$.
	
	Assume by contradiction that $0 = c_{\ee_i}(X) = \dim_\kk(V) - \dim_\kk(W)$.
	This yields that
	$$
	W \;\xrightarrow{\;\cong\;}\; s_0 \cdot  W \;=\; V
	$$ 
	and 
	$$
	W \;\xrightarrow{\;\cong\;}\; s_1 \cdot  W \;=\; V. 
	$$
	Dividing by $s_0$ and using the embedding $W \subset K(X)$, we obtain 
	$$
	W \;=\; r \cdot  W.
	$$
	Finally, the Cayley--Hamilton theorem yields the contradiction that $r$ is algebraic over $\kk$.
	Therefore, $c_{\ee_i}(X) >0$, and so the proof of part (iii) is complete. 
\end{proof}

\section{Standardization and support rigidity}\label{sec:standardization}
In this section, we use the process of standardization developed in
\cite{CCMM,CCC,CLM} to apply \autoref{thm:rigidity} to double
Grothendieck polynomials. More precisely, we use a slight variation of the construction in
\cite[Setup~7.1 and the proof of Proposition~7.9]{CCC}, and show that it
preserves rational singularities and the entire $K$-polynomial.

Let
$$
 R\;:=\;\kk\left[x_{i,j}\mid 1\leq i,j\leq n\right]
 \qquad \text{ and } \qquad 
 S\;:=\;\kk\left[w_{i,j},z_{i,j},w'_{i,j},z'_{i,j}
 \mid 1\leq i,j\leq n\right].
$$
Following the notation of \cite[Setup~4.1]{CCMM}, we consider $R$ and $S$
as $(\ZZ^n\oplus\ZZ^n)$-graded rings by setting
$$
 \deg(x_{i,j})=\ee_i\oplus\ee_j,\qquad
 \deg(w_{i,j})=\deg(w'_{i,j})=\ee_i\oplus\zero,\qquad
 \deg(z_{i,j})=\deg(z'_{i,j})=\zero\oplus\ee_j.
$$
We define the multigraded $\kk$-algebra homomorphism
\begin{equation}\label{eq:two-term}
 \phi:R\longrightarrow S,\qquad
 x_{i,j}\longmapsto w_{i,j}z_{i,j}+w'_{i,j}z'_{i,j}.
\end{equation}
For an $R$-homogeneous ideal $I\subset R$, we say that the extension
$J:=\phi(I)S\subset S$ is the \emph{standardization} of $I$ under $\phi$.
Let $\bt=(t_1,\ldots,t_n)$ and $\bs=(s_1,\ldots,s_n)$ be variables
indexing the two summands of the grading. Recall that, if $F_\bullet$ is a
finite multigraded free $R$-resolution of $R/I$ with
$F_p=\bigoplus_qR(-\ba_{p,q},-\bb_{p,q})$, then its \emph{$K$-polynomial} is
$$
 \K(R/I;\bt,\bs)\;:=\;
 \sum_{p,q}(-1)^p\bt^{\ba_{p,q}}\bs^{\bb_{p,q}}.
$$
The polynomial $\K(R/I;\one-\bx,\one-\by)$ is called the
\emph{twisted $K$-polynomial}. We use the same notation for quotients of
$S$, so that their $K$-polynomials belong to the same polynomial ring
$\ZZ[\bt,\bs]$.
The following proposition provides the properties of standardization that
we need (cf.~\cite[Proposition~4.2]{CCMM},
\cite[Theorem~7.2(i)--(iii)]{CCC}, and \cite[Theorem~3.4]{CLM}).
\begin{proposition}\label{prop:standardization}
The homomorphism $\phi$ is faithfully flat. Let $I\subset R$ be a proper
$R$-homogeneous ideal and $J=\phi(I)S$ be its standardization.
Then the following statements hold:
\begin{enumerate}[\rm (i)]
\item $\codim(I)=\codim(J)$.
\item $R/I$ and $S/J$ have the same multigraded Betti numbers. In particular,
$
 \K(R/I;\bt,\bs)=\K(S/J;\bt,\bs).
$
\item If $R/I$ is a domain with rational singularities, then $S/J$ is
also a domain with rational singularities.
\end{enumerate}
\end{proposition}
\begin{proof}
We first consider the homomorphism
$$
 \kk[x]\longrightarrow\kk[w,z,w',z'],\qquad x\longmapsto wz+w'z'.
$$
The ring $\kk[w,z,w',z']$ is torsion-free over the principal ideal domain $\kk[x]$, so this homomorphism is flat. 
Also, its fibers are nonempty, and hence it is faithfully flat. 
By taking the product of these morphisms, we obtain that $\phi$ is faithfully flat.
Then parts (i) and (ii) follow by the same arguments as in the proof of \cite[Theorem~7.2(i)--(iii)]{CCC}.

We now prove part (iii). 
Let $A:=R/I$ and $B:=S/J$, and suppose that $A$ is a domain with rational singularities.
 Since $B\cong S\otimes_RA$, the morphism 
 $$
f \;:\; \Spec(B) \;\longrightarrow\; \Spec(A)
 $$ 
 is obtained from $\Spec(S)\longrightarrow\Spec(R)$ by base change. 
 In particular, $f$ is faithfully flat. 
Consider a closed point $\Spec(\kk) \rightarrow \Spec(A)$. 
Let $a_{i,j}\in\kk$ be the image of $x_{i,j}$. 
We have
$$
 B\otimes_A\kk
 \;\cong\;S\otimes_R\kk
 \;\cong\;\bigotimes_{1\leq i,j\leq n}
 \frac{\kk[w,z,w',z']}{(wz+w'z'-a_{i,j})},
$$
where the tensor products on the right are taken over $\kk$.
Thus the closed fiber is a product of the integral quadrics
$Q_a:=V(wz+w'z'-a)\subset\mathbb A^4_\kk$.

If $a\neq0$, then $Q_a$ is smooth.
On the other hand, $Q_0$ has rational singularities (e.g., it is a normal toric variety).
Since rational singularities are preserved under products, we conclude that every closed fiber of $f : \Spec(B)\longrightarrow\Spec(A)$ is integral and has rational singularities.

Since $A$ is a domain and $B$ is flat over $A$, we have an injection
$
 B \hookrightarrow B\otimes_A\operatorname{Quot}(A).
$
The same quadric description over an algebraic closure of $\operatorname{Quot}(A)$ shows that the generic fiber is geometrically integral, and therefore $B$ is a domain (see \cite[\href{https://stacks.math.columbia.edu/tag/0BCM}{Tag 0BCM}]{stacks-project}). 
Finally, every closed point of $\Spec(B)$ maps to a closed point of $\Spec(A)$. 
By applying Elkik's theorem \cite[Th\'eor\`eme~5]{Elkik} at these closed points, we obtain that $B$ has rational singularities at every closed point, and hence everywhere.
This completes the proof of part (iii).
\end{proof}

We now apply this construction to Schubert determinantal ideals. 
For any permutation $w\in S_n$, let $I_w\subset R$ be the ideal generated by the minors of size $r_w(i,j)+1$ of the northwest $i\times j$ submatrix of the generic matrix $\big(x_{i,j}\big)_{1\leq i,j\leq n}$, where
$
 r_w(i,j):=\#\big\{k\leq i\mid w(k)\leq j\big\}.
$
By \cite[Theorem~A and Proposition~3.2.5]{KM}, the ideal $I_w$ is prime
of codimension $\ell(w)$. The ring $R/I_w$ has rational singularities (see \cite[Theorem~2.4.3 and the following paragraph]{KM}).
Moreover, by Knutson--Miller's formula \cite[Theorem~A]{KM}, we obtain
\begin{equation}\label{eq:fullK}
 \K(R/I_w;\one-\bx,\one-\by)\;=\;\G_w(\bx;\by).
\end{equation}
Here, we are translating Knutson--Miller's grading \cite{KM} to our positive grading in this section (see \cite[Lemma~3.3]{CCMM}).

\begin{definition}
	Let $\mathcal{S}_w:=\Supp(\G_w(\bx;\by)) \subset \NN^{2n}$ be the support of the double Grothendieck polynomial $\G_w$.
	Let $B_w$ be the subset of $\mathcal{S}_w$ of minimum total degree (i.e., $B_w$ is the support of the double Schubert polynomial of $w$).
	Let $T_w$ be the subset of $\mathcal{S}_w$ of maximum total degree.
\end{definition}

The following proposition gives the support properties that we need in the proof of the main theorem.

\begin{proposition}\label{prop:full-box}
The following statements hold:
\begin{enumerate}[\rm (i)]
\item Every $\ba\in\mathcal S_w$ satisfies $\bb\leq\ba$ for some
$\bb\in B_w$.
\item Let $\ba\in\mathcal S_w$, $\bb\in B_w$, and
$\ba'\in\NN^{2n}$ be such that $\bb\leq\ba'\leq\ba$.
Then $\ba'\in\mathcal S_w$.
\end{enumerate}
\end{proposition}
\begin{proof}
Let $J_w:=\phi(I_w)S$ be the standardization of $I_w$, and set $A_w:=S/J_w$. 
Let $m:=2n-1$, $\bm:=(m,\ldots,m)\in\NN^{2n}$ and $\PP:=(\PP_\kk^m)^{2n}$.
By \autoref{prop:standardization}(iii), $A_w$ is a domain with rational singularities. 
From \autoref{prop:standardization}(ii) and \autoref{eq:fullK}, we obtain
$$
 \K(A_w;\one-\bx,\one-\by)\;=\;\G_w(\bx;\by).
$$
Let $Y_w:=\MultiProj(A_w)\subset\PP$. 
Then $Y_w$ is integral with rational singularities, and
$$
 d \;:=\; \dim(Y_w) \;=\; \dim(A_w)-2n \;=\; |\bm|-\ell(w).
$$

We next compare the coefficients of the twisted $K$-polynomial with the Hilbert
coefficients of $Y_w$.
We write 
$$
 \G_w(\bx;\by)\;=\;
 \sum_{\zero\leq\ba = (\ba_x, \ba_y)}g_{\ba}\bx^{\ba_x}\by^{\ba_y}.
$$
The Hilbert series of $A_w$ is therefore
$$
 \operatorname{Hilb}_{A_w}(\bt,\bs)
 \;=\;
 \frac{\G_w(\one-\bt;\one-\bs)}
 {\prod_{i=1}^n(1-t_i)^{m+1}\prod_{j=1}^n(1-s_j)^{m+1}}.
$$
By \cite[Theorem~B]{KM}, the ideal $I_w \subset R$ has a square-free monomial initial ideal.
Since $r_w(n,j)=j$ and $r_w(i,n)=i$, the generators of $I_w$ do not involve the last row or column, and neither do the minimal generators of this initial ideal.
The Taylor resolution of the quotient by this initial ideal has square-free shifts, each involving at most $n-1$ variables from any row or column.
Since Gr\"obner degeneration preserves the $K$-polynomial, $\G_w$ has degree at most $n-1$ in each variable.
As $n-1\leq m$, we have $g_{\ba}=0$ unless $\zero\leq\ba\leq\bm$.
Expanding the Hilbert series as in the proof of \cite[Lemma~7.2]{KMF}, we obtain
$$
 P_{Y_w}(\bt,\bs)\;=\;
 \sum_{\zero\leq\ba=(a_1,\ldots,a_{2n})\leq\bm}g_{\ba}
 \prod_{i=1}^n\binom{t_i+m-a_i}{m-a_i}
 \prod_{j=1}^n\binom{s_j+m-a_{n+j}}{m-a_{n+j}}.
$$
By comparing this expression with the binomial expansion in \autoref{sec:rigidity}, we obtain
\begin{equation}\label{eq:K-Hilbert-coefficients}
 g_{\ba}\;=\;e_{\bm-\ba}(Y_w)
 \;=\;(-1)^{|\ba|-\ell(w)}c_{\bm-\ba}(Y_w).
\end{equation}
In particular,
$$
 \mathcal S_w\;=\;\bm-\hsupp_\PP(Y_w)\qquad \text{ and } \qquad
 B_w\;=\; \bm-\msupp_\PP(Y_w).
$$
Finally, the claims of the proposition follow from \autoref{thm:rigidity}.
\end{proof}

We conclude this section by comparing the supports of double and ordinary Grothendieck polynomials. 
This comparison is motivated by \cite[Corollary 2.5]{PSW}.
Let
$$
 G_w^+(\bx;\by)\;:=\;(-1)^{\ell(w)}\G_w(-\bx;-\by) \;=\; (-1)^{\ell(w)}  \K(R/I_w;\one+\bx,\one+\by).
$$
By \autoref{eq:K-Hilbert-coefficients} and \autoref{thm:rigidity}(i),
the coefficients of $G_w^+(\bx;\by)$ are nonnegative. Moreover,
$\Supp(G_w^+(\bx;\by))=\Supp(\G_w(\bx;\by))$. We also set
$G_w^+(\bx):=G_w^+(\bx;\zero)$.

\begin{lemma}\label{lem:ordinary-domination}
Let $\ba = (\ba_x,\ba_y)\in \mathcal{S}_w = \Supp(\G_w(\bx;\by))$. 
Then there exist
$$
 \mathbf p\in\Supp(\G_w(\bx))
 \qquad\text{and}\qquad
 \mathbf q\in\Supp(\G_{w^{-1}}(\by))
$$
such that $\ba_x\leq\mathbf p$ and $\ba_y\leq\mathbf q$.
If $\ba = (\ba_x,\ba_y)$ has minimum total degree in the double support,
then $\mathbf p$ and $\mathbf q$ can be chosen to have minimum total
degree in the corresponding ordinary supports.
\end{lemma}
\begin{proof}
Since $\deg(x_{i,j})=\ee_i\oplus\ee_j \in \NN^n \oplus \NN^n$, every shift $(\br,\bc) \in \NN^n \oplus \NN^n$ in the minimal multigraded free resolution of $R/I_w$ satisfies $|\br|=|\bc|$.
Thus, for an indeterminate $u$, we have
$$
 \K(R/I_w;\bt,u\one)\;=\;\K(R/I_w;u\bt,\one),
$$
since each term $\bt^{\br}u^{|\bc|}$ on the left equals $(u\bt)^{\br}$ on the right.
Take $\bt=\one+\bx$ and $u=1+z$.
Then \autoref{eq:fullK} and the definition of $G_w^+$ give
$$
\begin{aligned}
 G_w^+(\bx;z\one)
 &\;=\;(-1)^{\ell(w)}\K(R/I_w;\one+\bx,(1+z)\one)\\
 &\;=\;(-1)^{\ell(w)}\K(R/I_w;(1+z)(\one+\bx),\one)\\
 &\;=\;G_w^+(x_1+z+x_1z,\ldots,x_n+z+x_nz),
\end{aligned}
$$
where the last equality uses $(1+z)(1+x_i)=1+(x_i+z+x_iz)$.
By nonnegativity, $\bx^{\ba_x}z^{|\ba_y|}$ occurs in
$\prod_{i=1}^n(x_i+z+x_iz)^{p_i}$ for some
$\mathbf p\in\Supp(G_w^+(\bx))$.
Every monomial in this product has $x_i$-degree at most $p_i$ for each $i$
and total degree at least $|\mathbf p|$.
Hence $\ba_x\leq\mathbf p$ and $|\mathbf p|\leq|\ba_x|+|\ba_y|=|\ba|$.

Suppose now that $d:=|\ba|$ is the minimum total degree in $\mathcal S_w$.
Since $\G_w(\bx)=\G_w(\bx;\zero)$, every
$\br\in\Supp(\G_w(\bx))$ satisfies $(\br,\zero)\in\mathcal S_w$,
and therefore $|\br|\geq d$.
In particular, $d\leq|\mathbf p|\leq|\ba|=d$, so $|\mathbf p|=d$.
Thus $\mathbf p$ has minimum total degree in the ordinary support.

Finally, by \cite[Corollary~6.6]{LRS}, we have
$
 \G_w(\bx;\by)=\G_{w^{-1}}(\by;\bx).
$
Thus $(\ba_y,\ba_x)\in\mathcal S_{w^{-1}}$.
Applying the preceding argument to this vector gives
$\mathbf q\in\Supp(\G_{w^{-1}}(\by))$ with $\ba_y\leq\mathbf q$.
If $\ba$ has minimum total degree in $\mathcal S_w$, then
$(\ba_y,\ba_x)$ has minimum total degree in $\mathcal S_{w^{-1}}$,
so $\mathbf q$ also has minimum total degree in its ordinary support.
\end{proof}

\section{Proof of the main theorem}\label{sec:proof}

In this section, we prove \autoref{thm:main}.
Before proceeding with the proof we need to recall further results regarding Grothendieck polynomials.

Fix $w\in S_n$. 
We keep the notation from  \autoref{sec:standardization}. 
In particular, $\mathcal S_w = \Supp(\G_w(\bx;\by)) \subset \NN^{2n}$ is the support of the double Grothendieck polynomial $\G_w$.
Also, $B_w$ and  $T_w$ are the subsets of $\mathcal S_w$ of minimum total degree and maximum total degree, respectively.
For $v=w$ or $v= w^{-1}$, let $\mathcal A_v:=\Supp(\G_v(\bx)) \subset \NN^n$ be the support of the ordinary Grothendieck polynomial, and $M_v$ be its subset of maximum total degree.

We use the following two results, the first of which is proved geometrically in \autoref{sec:ordinary}.

\begin{theorem}[{Nguyen-Dang--Wang~\cite[Theorem 1.1]{NW}}]\label{prop:ordinary-exchange}
For every $w\in S_n$, the support of $\G_w(\bx)$ is an $M^\natural$-convex set.
\end{theorem}

\begin{theorem}[{Pechenik--Speyer--Weigandt~\cite[Theorem~1.5]{PSW}}]\label{thm:top-support}
For every $w \in S_n$, we have the equality $T_w = M_w\times M_{w^{-1}}$.
\end{theorem}

Finally, we are ready to prove our main result.

\begin{proof}[Proof of \autoref{thm:main}]
By \autoref{prop:ordinary-exchange}, each $\mathcal A_v$ is $M^\natural$-convex.
As a consequence, $M_v$ is $M$-convex and every vector in $\mathcal A_v$ is bounded above by a vector in $M_v$ (see \autoref{sec:poly}).
By \autoref{thm:top-support}, $T_w=M_w\times M_{w^{-1}}$, and so $T_w$ is $M$-convex.
Also, $B_w$ is the support of a double Schubert polynomial, and so it is
$M$-convex by \cite[Theorem~A]{CCMM}.

We first describe $\mathcal S_w$ in terms of $B_w$ and $T_w$.
Let $\ba=(\ba_x,\ba_y)\in\mathcal S_w$.
By \autoref{lem:ordinary-domination}, there are $\mathbf p\in\mathcal A_w$ and $\mathbf q\in\mathcal A_{w^{-1}}$ with $\ba_x\leq\mathbf p$ and $\ba_y\leq\mathbf q$.
Choosing vectors in $M_w$ and $M_{w^{-1}}$ above $\mathbf p$ and $\mathbf q$ gives $\ba\leq\bt$ for some $\bt\in T_w$.
By \autoref{prop:full-box}(i), there is also $\bb\in B_w$ with $\bb\leq\ba$.
Conversely, if $\bb\leq\ba\leq\bt$ for $\bb\in B_w$ and $\bt\in T_w$, then \autoref{prop:full-box}(ii) gives $\ba\in\mathcal S_w$.
Therefore, we obtain the equality
$$
 \mathcal{S}_w \;=\; \big\{\ba\in\NN^{2n}\mid\bb\leq\ba\leq\bt
       \text{ for some }\bb\in B_w,\ \bt\in T_w\big\}.
$$

Let
$E_x:=[n]$ and $E_y:=[2n]\setminus[n]=\{n+1,\ldots,2n\}$ be the two coordinate blocks.
By \autoref{rem:intermediate-support}, it suffices to show that the function $\rk_{T_w}-\rk_{B_w} : 2^{[2n]} \rightarrow \ZZ$ is nondecreasing under inclusion.

Let $L_v$ be the subset of $\mathcal A_v$ of minimum degree $\ell(v)$.
By \autoref{lem:ordinary-domination}, every $\bb = (\bb_x, \bb_y) \in B_w$ satisfies $\bb_x\leq\mathbf p$ for some $\mathbf p\in L_w$.
By specialization, we obtain that $(\mathbf p,\zero)\in B_w$ for every $\mathbf p\in L_w$.
Together with the product formula for $T_w$, this gives
$$
 \rk_{B_w}(I) \;=\; \rk_{L_w}(I)\quad \text{and} \quad
 \rk_{T_w}(I) \;=\; \rk_{M_w}(I) \quad \text{ for all \;}
 I\subseteq E_x.
$$
By symmetry, the same argument yields
$$
\rk_{B_w}(I) \;=\; \rk_{L_{w^{-1}}}(I)\quad \text{and} \quad
\rk_{T_w}(I) \;=\; \rk_{M_{w^{-1}}}(I) \quad \text{ for all \;}
I\subseteq E_y.
$$
Therefore, by combining \autoref{prop:ordinary-exchange} and \autoref{rem:intermediate-support}, we obtain that the function  $\rk_{T_w}-\rk_{B_w} : 2^{[2n]} \rightarrow \ZZ$ is nondecreasing under inclusion when restricted to either $2^{E_x}$ or $2^{E_y}$.

Now let $I\subseteq[2n]$ and choose $i\in[2n]\setminus I$. 
Let $E\in\{E_x,E_y\}$ be the block containing $i$.
Let $J:=I\cap E$ and $K := I \setminus E$.
Submodularity of $\rk_{B_w}$, the preceding observation, and the product description of $T_w$ give
$$
\begin{aligned}
 \rk_{B_w}(I\cup\{i\})-\rk_{B_w}(I)
 &\;\leq\;\rk_{B_w}(J\cup\{i\})-\rk_{B_w}(J)\\
 &\;\leq\;\rk_{T_w}(J\cup\{i\})-\rk_{T_w}(J)\\
 &\;=\; \Big(\rk_{T_w}(J\cup\{i\}) + \rk_{T_w}(K) \Big)- \Big(\rk_{T_w}(J) + \rk_{T_w}(K)\Big)\\
 &\;=\;\rk_{T_w}(I\cup\{i\})-\rk_{T_w}(I).
\end{aligned}
$$
Finally, we have that $\rk_{T_w}-\rk_{B_w}$ is nondecreasing under inclusion, and
\autoref{rem:intermediate-support} completes the proof of the theorem. 
\end{proof}

\appendix

\section{Ordinary Grothendieck Lorentzianity}\label{sec:ordinary}
This appendix grew out of the author's efforts to understand the recent result of Nguyen-Dang and Wang~\cite{NW} showing that the positive homogenizations of ordinary Grothendieck polynomials are denormalized Lorentzian.
Their result settled positively a conjecture of Huh, Matherne, M\'esz\'aros, and St.~Dizier~ \cite{HMMSD}.
Following their Bott--Samelson and bundle constructions, we give a simplified proof of their Lorentzianity result for ordinary Grothendieck polynomials.
Our aim is to make the argument as self-contained and direct as possible.
Since our proof of \autoref{thm:main}  relies on \autoref{prop:ordinary-exchange}, we include this appendix so that the paper contains a complete argument for it.
We also hope that this presentation will be useful to other researchers in algebraic combinatorics.

Given a variety $Y$ and a vector bundle $\EE$  on $Y$, we use the convention that 
$$
\PP_Y(\mathscr{E}) \;:=\; \fProj_Y\left(\Sym\left(\EE^\vee\right)\right)
$$ parametrizes lines in
$\mathscr{E}$, and that $c_z(\EE) := \sum_{i \ge 0} c_i(\EE)z^i$ is the Chern polynomial of $\EE$.

We now recall the construction of relative Bott--Samelson varieties following 
Fulton--Pragacz~\cite[Appendix~C]{FP} (see also
\cite{Magyar,BrionFlags,Hudson,Li,AF,KW}).
Fix a permutation $w\in S_n$ on $[n]=\{1,\ldots,n\}$. 
Let
$$
 N\;:=\;\binom n2,\qquad
 u \;:=\; w_0w^{-1}\qquad \text{and} \qquad L\;:=\;\ell(u)=N-\ell(w).
$$
For $1\leq a<n$, define
$$
 c_a \;:=\; \#\Big\{b>a\mid u^{-1}(b)<u^{-1}(a)\Big\}\qquad \text{and} \qquad
 B_a \;:=\; s_as_{a+1}\cdots s_{a+c_a-1},
$$
where $B_a$ is the identity if $c_a=0$.
 Let 
 $$
 \mathbf{i} \;:=\; \left(i_1,\ldots,i_L\right) \;:=\; \big(
 \ldots, \, \underbrace{a, a+1,\ldots,a+c_a-1}_{c_a \text{ indices}},\, \ldots,\, \underbrace{1, 2,\ldots,c_1}_{c_1 \text{ indices}}
 \big)
 $$ 
 be the sequence of indices in the displayed product $B_{n-1}B_{n-2}\cdots B_1$, read from left to right.
Therefore
$$
 B_{n-1}B_{n-2}\cdots B_1 \;=\; s_{i_1}s_{i_2}\cdots s_{i_L} \;=\; u.
$$
Moreover, $\sum_ac_a=\ell(u)$, and so $\mathbf{i}=(i_1,\ldots,i_L)$ is a reduced word of $u$.

Let $\PP:=(\PP_\kk^{n-1})^n$, $\EE:=\bigoplus_{a=1}^n\OO_{\PP}(\ee_a)$ and $h_a=c_1(\OO_{\PP}(\ee_a)) \in A^1(\PP)$.
Starting with $Z_0:=\PP$, take the full quotient flag
$$
\EE_\bullet^{(0)} \;:\; \EE=\mathscr{E}_n^{(0)} \;\twoheadrightarrow\;
 \mathscr{E}_{n-1}^{(0)} \;\twoheadrightarrow \; \cdots \; \twoheadrightarrow \;
 \mathscr{E}_1^{(0)} \; \twoheadrightarrow \; \mathscr{E}_0^{(0)}=0
 $$
with
$
 \mathscr{E}_j^{(0)} := \bigoplus_{a=1}^j\OO_{\PP}(\ee_a).
$
Suppose we have constructed $Z_{t-1}$ and a full quotient flag $\EE_\bullet^{(t-1)}$ on $Z_{t-1}$.
Let
$$
 \mathscr{K}_t \;:=\; \Ker\bigl(
 \mathscr{E}_{i_t+1}^{(t-1)}
 \twoheadrightarrow\mathscr{E}_{i_t-1}^{(t-1)}\bigr)
 \qquad \text{and} \qquad
\rho_t \;: \; Z_t:=\PP_{Z_{t-1}}(\mathscr{K}_t) \;\longrightarrow\; Z_{t-1}.
$$
We have the universal short exact sequence
\begin{equation}
	\label{eq_univ_ex_seq}
	 0 \;\longrightarrow\; \OO_{Z_t}(-1) \;\longrightarrow\;
	\rho_t^*(\mathscr{K}_t) \;\longrightarrow\; \mathscr{Q}_t \;\longrightarrow\; 0
\end{equation}
with $\mathscr{Q}_t:=\rho_t^*(\mathscr{K}_t)/\OO_{Z_t}(-1)$.
We define the full quotient flag $\EE_\bullet^{(t)}$ on $Z_t$ given by 
\begin{equation}
	\label{eq_def_EE}
	\EE_j^{(t)} \;:=\;  
	\begin{cases}
		\rho_t^*\big(\EE_j^{(t-1)}\big) & \quad \text{if \;} j \neq i_t \\
		\rho_t^*\big(\mathscr{E}_{i_t+1}^{(t-1)}\big)/\OO_{Z_t}(-1) & \quad \text{if \;} j = i_t.
	\end{cases}
\end{equation}
We have the two short exact sequences 
\begin{equation}\label{eq:ordinary-step-sequences_q_t}
 0 \;\longrightarrow\; \OO_{Z_t}(-1)\longrightarrow
 \EE_{i_t+1}^{(t)}=\rho_t^*(\mathscr{E}_{i_t+1}^{(t-1)}) \;\longrightarrow\;
 \mathscr{E}_{i_t}^{(t)} \;\longrightarrow \;0
 \end{equation}
 and 
 \begin{equation}\label{eq:ordinary-step-sequences_p_t}
 0\;\longrightarrow\;\mathscr{Q}_t \;\longrightarrow\;
 \mathscr{E}_{i_t}^{(t)} \; \longrightarrow \;
 \rho_t^*(\mathscr{E}_{i_t-1}^{(t-1)})= \EE_{i_t-1}^{(t)} \;\longrightarrow \; 0.
\end{equation}
Finally, consider the composition
$$
 \rho \;:=\; \rho_1\circ\rho_2\circ\cdots\circ\rho_L \;:\; Z_w:=Z_L \;\longrightarrow\; \PP.
$$
Since each $\rho_t$ is a $\PP^1$-bundle, we have that $Z_w$ is smooth, integral, and projective, with $\dim(Z_w)=\dim(\PP)+L=2N+L$.
Every $\mathscr{E}_j^{(t)}$ is a quotient of the pullback of
$\mathscr{E}$, so all these bundles are globally generated.

We consider the Chern roots
\begin{equation}
	\label{eq_def_roots}
	 r_k^{(t)} \;:=\; c_1\!\left(\Ker\bigl(
	\mathscr{E}_k^{(t)}\twoheadrightarrow\mathscr{E}_{k-1}^{(t)}\bigr)\right) \;\in\; A^1(Z_t),
	 \qquad
	\mathbf{r}^{(t)} \;:=\; (r_1^{(t)},\ldots,r_n^{(t)}).
\end{equation}
At step $t$, these Chern classes satisfy
\begin{equation}
	\label{eq_compatible_roots}
	 r_k^{(t)} \;=\; \rho_t^*(r_k^{(t-1)})
	\quad\text{for }k \;\notin\; \{i_t,i_t+1\}.
\end{equation}
To simplify notation, we set 
$$
p_t \;:=\; r_{i_t}^{(t)} \;=\; c_1(\mathscr{Q}_t) \qquad \text{ and }
\qquad q_t \;:=\; r_{i_t+1}^{(t)} \;=\; c_1(\OO_{Z_t}(-1)).
$$
We shall also suppress the pullbacks of Chern classes (for instance, we may write $r_k^{(t-1)} \in A^1(Z_t)$, instead of $\rho_t^*(r_k^{(t-1)}) \in A^1(Z_t)$).

The following classical lemma encodes the operators $\partial_j$ in terms of pushforwards along $\rho_t$.
For completeness, we show a direct short argument in our setting.

\begin{lemma}[{\cite[Appendix~C, Proposition~(1)]{FP}}]\label{lem:ordinary-FP}
Fix $1\leq t\leq L$, and let
$
 f\in \left(A^*(Z_{t-1})[z]\right)\left[x_{i_t},x_{i_t+1}\right],
$
where $\partial_{i_t}$ acts on the variables $x_{i_t}$ and $x_{i_t+1}$. 
Then
$$
 (\rho_t)_*\left(f(p_t,q_t)\right)
 \;=\; \partial_{i_t}(f)\bigl(
 r_{i_t}^{(t-1)},r_{i_t+1}^{(t-1)}\bigr).
$$
\end{lemma}
\begin{proof}
By construction, we have that 
$
 c_1(\mathscr{K}_t)=r_{i_t}^{(t-1)}+r_{i_t+1}^{(t-1)}
 $
 and 
 $
 c_2(\mathscr{K}_t)=r_{i_t}^{(t-1)}r_{i_t+1}^{(t-1)}.
$
On the other hand, the universal sequence \autoref{eq_univ_ex_seq} yields $c_1(\mathscr{K}_t) = p_t + q_t$.
Then
 the projective bundle formula for Chow rings
\cite[Theorem~3.3, Example~8.3.4]{FultonIT}  gives
$$
 A^*(Z_t) \;\cong\;
 \frac{A^*(Z_{t-1})[q_t]}
 {\bigl(q_t^2-c_1(\mathscr{K}_t)q_t+c_2(\mathscr{K}_t)\bigr)}
 \;=\; \frac{A^*(Z_{t-1})[q_t]}
 {\bigl((q_t-r_{i_t}^{(t-1)})(q_t-r_{i_t+1}^{(t-1)})\bigr)}.
$$
In particular, $A^*(Z_t)$ is free over $A^*(Z_{t-1})$ with basis $1,q_t$.
Then \cite[Proposition~3.1(a)]{FultonIT} yields
\begin{equation}
	\label{eb_push_PP1_bundle}
	 (\rho_t)_*(1)=0,\qquad
	(\rho_t)_*\left(c_1(\mathscr{K}_t)\right)
	=0,
	\qquad
	(\rho_t)_*(-q_t)=1, \qquad (\rho_t)_*(p_t)=1.
\end{equation}
Consider the polynomial 
$$
H(T) \;:=\; f(T,x_{i_t}+x_{i_t+1}-T)
- f(x_{i_t},x_{i_t+1}) - (T-x_{i_t})\partial_{i_t}(f) \;\in\; \left(A^*(Z_{t-1})[z,x_{i_t}, x_{i_t+1}]\right)[T],
$$
where $T$ is a new indeterminate. 
Since $H(x_{i_t}) = H(x_{i_t+1})=0$, it follows that $H(T)$ is divisible by $(T-x_{i_t})(T-x_{i_t+1})$.
By specializing $T=p_t$, $x_{i_t} = r_{i_t}^{(t-1)}$ and $x_{i_t+1} = r_{i_t+1}^{(t-1)}$, we obtain the equality
$$
 f(p_t,q_t)
\;=\; f\bigl(r_{i_t}^{(t-1)},r_{i_t+1}^{(t-1)}\bigr)
+(p_t-r_{i_t}^{(t-1)})
\partial_{i_t}(f)\bigl(
r_{i_t}^{(t-1)},r_{i_t+1}^{(t-1)}\bigr) \;\in\; A^*(Z_t)[z],
$$
because $q_t = r_{i_t}^{(t-1)}+r_{i_t+1}^{(t-1)}-p_t$ and $(p_t-r_{i_t}^{(t-1)})(p_t-r_{i_t+1}^{(t-1)})=(q_t-r_{i_t}^{(t-1)})(q_t-r_{i_t+1}^{(t-1)})=0$ in $A^*(Z_t)$.
Finally, \autoref{eb_push_PP1_bundle} and the projection formula give
$$
\begin{aligned}
 (\rho_t)_*\left(f(p_t,q_t)\right)
 ={}&f\bigl(r_{i_t}^{(t-1)},r_{i_t+1}^{(t-1)}\bigr)(\rho_t)_*(1) +\partial_{i_t}(f)\bigl(
 r_{i_t}^{(t-1)},r_{i_t+1}^{(t-1)}\bigr)
 \left((\rho_t)_*(p_t)
 -r_{i_t}^{(t-1)}(\rho_t)_*(1)\right)\\
 ={}&\partial_{i_t}(f)\bigl(
 r_{i_t}^{(t-1)},r_{i_t+1}^{(t-1)}\bigr).
\end{aligned}
$$
This concludes the proof of the lemma.
\end{proof}

The \emph{normalization operator} $\mathrm{N}$ on $\mathbb{R}[t_1,\ldots,t_m]$ is the linear operator determined by
$$
 \mathrm{N}\left(t_1^{a_1}\cdots t_m^{a_m}\right)
 \;:=\;\frac{t_1^{a_1}\cdots t_m^{a_m}}{a_1!\cdots a_m!}
 \qquad\text{for \;\; }\ba=(a_1,\ldots,a_m)\in\NN^m.
$$
A homogeneous polynomial $f$ is \emph{denormalized Lorentzian} if $\mathrm{N}(f)$ is \emph{Lorentzian} in the sense of Br\"and\'en and Huh \cite{BH}.

\Needspace{6\baselineskip}
We are now ready for the main result of this appendix. 
Throughout the theorem below and its proof, we keep the same notation as above. 

\begin{theorem}[{Nguyen-Dang--Wang~\cite{NW}}]\label{prop:ordinary-support}
Let $w\in S_n$, and let
$G_w^+(\bx)=(-1)^{\ell(w)}\G_w(-\bx;\zero)$ be the sign-corrected ordinary Grothendieck polynomial. 
Then the homogenization
$$
 F_w(\bx,z) \;:=\; z^N G_w^+\left(\frac{x_1}{z},\ldots,\frac{x_n}{z}\right)
$$
is a denormalized Lorentzian polynomial.
In particular, $\Supp(F_w)$ is $M$-convex, and $\Supp(G_w^+)$ is $M^\natural$-convex.
\end{theorem}
\begin{proof}
For $1\leq i<n$, define operators on
$\ZZ[\bx,z]$ given by
$$
 D_i(f) \;:=\; \partial_i\bigl((x_{i+1}+z)f\bigr)
 \qquad \text{ and } \qquad
 \Theta_i(f):=\partial_i\bigl(x_{i+1}(x_i+z)f\bigr),
$$
where $s_i$ fixes $z$. 
Let
$\delta:=(n-1,n-2,\ldots,0)$. 
Since
$\mathbf i=(i_1,\ldots,i_L)$ is reduced for $u=w_0w^{-1}$, we obtain
$$
 w^{-1}w_0 \;=\; w_0uw_0
 \;=\; s_{n-i_1}\cdots s_{n-i_L}.
$$
Therefore sign-correction and homogenization of the recursion in
\autoref{def:grothendieck} give
\begin{equation}\label{eq:ordinary-D-recursion}
 F_w \;=\; \bigl(D_{n-i_1}\cdots D_{n-i_L}\bigr)(\bx^\delta).
\end{equation}
In particular, $F_w$ is a homogeneous polynomial of degree $N$.
We extend these operators to Laurent polynomials and define the involution
$$
 \iota(f)(\bx,z) \;:=\; (x_1\cdots x_n)^{n-1}
 f(x_n^{-1},\ldots,x_1^{-1},z^{-1}).
$$
We have that $\iota(\bx^\delta)=\bx^\delta$.
In \autoref{rem:ordinary-conjugation}, we verify the following identity
$$
 z(\iota D_{n-i}\iota)(f)
 \;=\;\frac{x_{i+1}(x_i+z)f-x_i(x_{i+1}+z)s_i(f)}{x_i-x_{i+1}}
 \;=\; \Theta_i(f).
$$
Then applying these identities successively to
\autoref{eq:ordinary-D-recursion} yields
\begin{equation}\label{eq:ordinary-reciprocity}
\begin{aligned}
 P_w(\bx,z)
 &\;:=\;\bigl(\Theta_{i_1}\cdots\Theta_{i_L}\bigr)(\bx^\delta) \\
  &\;=\;z^L\iota(F_w)(\bx,z)\\
 &\;=\; (x_1\cdots x_n)^{n-1} z^{-\ell(w)}
 G_w^+(z/x_n,\ldots,z/x_1).
\end{aligned}
\end{equation}
Consider the vector bundle $\mathscr{H}=\bigoplus_{j=1}^{n-1}\mathscr{E}_j^{(L)}$ on $Z_w$.
Since $\mathscr{E}_j^{(L)}$ has $r_1^{(L)},\ldots,r_j^{(L)}$ as Chern roots, we obtain
$$
 c_N(\mathscr{H})
 \;=\; \prod_{j=1}^{n-1}c_j(\mathscr{E}_j^{(L)})
 \;=\;\prod_{k=1}^n(r_k^{(L)})^{n-k}
 \;=\;\bx^\delta\big|_{\bx=\mathbf{r}^{(L)}}.
$$
Since
$\mathbf{r}^{(0)}=(h_1,\ldots,h_n)$, by successively applying \autoref{lem:ordinary-FP} and the projection formula, we obtain
\begin{equation}\label{eq:ordinary-pushforward}
 \rho_*\left(c_N(\mathscr{H})
       \prod_{t=1}^L q_t(p_t+z)\right)
 \;=\; P_w(\bh,z).
\end{equation}

We now express the class being pushed forward in terms of globally generated bundles. 
For the remainder of the proof, every bundle on an intermediate stage $Z_t$ is understood to be pulled back to $Z_w$, and we keep the same notation for its pullback. 

Let $d_a:=\sum_{b=a}^{n-1}c_b$ for $1\leq a\leq n$; in particular, $d_1=L$ and $d_n=0$.
For rows with $c_a=0$, we set $\mathscr{A}_a:=0$ and $\mathscr{B}_a:=0$.

We analyze each nonempty word $B_a = s_as_{a+1}\cdots s_{a+c_a-1}$.
For $1\leq j\leq c_a$, \autoref{eq:ordinary-step-sequences_p_t} gives
$$
 0 \;\longrightarrow\; \mathscr{Q}_{d_{a+1}+j}
 \;\longrightarrow\; \EE_{a+j-1}^{(d_{a+1}+j)}
 \;\longrightarrow\; \EE_{a+j-2}^{(d_{a+1}+j-1)}
 \;\longrightarrow\; 0.
$$
By definition \autoref{eq_def_EE}, we have that $\EE_{a-1}^{(d_a-c_a)} = \EE_{a-1}^{(0)}$.
It follows that the classes
$$
p_{d_{a+1}+j} \; = \; r_{a+j-1}^{(d_{a+1}+j)},
\qquad 1\leq j\leq c_a,
$$ 
are the Chern roots of the vector bundle $\mathscr{A}_a := \Ker\left(\EE_{a+c_a-1}^{(d_a)} \twoheadrightarrow \EE_{a-1}^{(0)}\right)$.
Since the initial flag $\EE_\bullet^{(0)}$ splits at each step, we obtain a natural surjection $\bigoplus_{b=a}^n \OO_\PP(\ee_b)\cong \EE/\EE_{a-1}^{(0)} \twoheadrightarrow \mathscr{A}_a$, and so $\mathscr{A}_a$ is globally generated.
Thus we obtain the globally generated vector bundle 
\begin{equation}\label{eq:ordinary-row_p_t}
	\mathscr{A} \;:=\; \bigoplus_{a=1}^{n-1} \mathscr{A}_a \qquad \text{ such that \quad  $\rank(\mathscr{A}) \;=\; L$ \quad and \quad }   c(\mathscr{A})
	\;=\; \prod_{t=1}^L(1+p_t).
\end{equation}

Similarly, from the natural surjections $\EE_{a+c_a}^{(d_{a+1})} \twoheadrightarrow \EE_{a+c_a-1}^{(d_{a+1})} \twoheadrightarrow \cdots \twoheadrightarrow \EE_{a}^{(d_{a+1})} = \EE_{a}^{(0)}$, we obtain the globally generated vector bundle 
$$
\mathscr{B}_a
\;:=\;\Ker\left(
\mathscr{E}_{a+c_a}^{(d_{a+1})}
\twoheadrightarrow\mathscr{E}_a^{(0)}\right)
$$
with Chern roots $r_{a+j}^{(d_{a+1})}$ for $1\leq j\leq c_a$.
For $t=d_{a+1}+j$ with $1\leq j\leq c_a$, we have $i_t=a+j-1$, and so by successively applying \autoref{eq_compatible_roots}, we get
$$
r_{i_t+1}^{(t-1)} \;=\; r_{i_t+1}^{(t-2)} \;=\; \cdots \;=\; r_{i_t+1}^{(d_{a+1}+1)} \;=\; r_{a+j}^{(d_{a+1})}.
$$
Thus we obtain the globally generated vector bundle
\begin{equation}\label{eq:ordinary-row_q_t}
 \mathscr{B}\;:=\;\bigoplus_{a=1}^{n-1}\mathscr{B}_a,
 \qquad \text{ such that } \qquad \rank(\mathscr{B})=L,
 \qquad
 c(\mathscr{B})\;=\;\prod_{t=1}^L
 \left(1+r_{i_t+1}^{(t-1)}\right).
\end{equation}
To relate these roots to the factors $q_t$ in
\autoref{eq:ordinary-pushforward}, for each $0 \le t \le L$, we set
$$
 \mathscr{H}_t \; := \; \bigoplus_{j=1}^{n-1}\mathscr{E}_j^{(t)}.
$$
In particular, $\mathscr{H}_L=\mathscr{H}$.
At step $t$, \autoref{eq:ordinary-step-sequences_q_t} and \autoref{eq_def_roots} give the equalities
$$
 c\bigl(\mathscr{E}_{i_t}^{(t)}\bigr)(1+q_t)
 \;=\;c\bigl(\mathscr{E}_{i_t+1}^{(t-1)}\bigr)
 \;=\;c\bigl(\mathscr{E}_{i_t}^{(t-1)}\bigr)
 \left(1+r_{i_t+1}^{(t-1)}\right).
$$
By multiplying all these equalities for $t = 1,\ldots, L$ and applying \autoref{eq:ordinary-row_q_t}, we obtain
$$
 c(\mathscr{H})\prod_{t=1}^L(1+q_t)
 \;=\;c(\mathscr{H}_0)c(\mathscr{B}).
$$
Since $\rank(\mathscr{H})=\rank(\mathscr{H}_0)=N$ and
$\rank(\mathscr{B})=L$, it follows that
\begin{equation}\label{eq:ordinary-top-Chern}
 c_N(\mathscr{H})\prod_{t=1}^L q_t
 \;=\;c_N(\mathscr{H}_0)c_L(\mathscr{B}).
\end{equation}
Combining \autoref{eq:ordinary-pushforward},
\autoref{eq:ordinary-row_p_t}, and \autoref{eq:ordinary-top-Chern} yields
\begin{equation}\label{eq:ordinary-Chern-identity}
 P_w(\bh,z)=
 \rho_*\left(c_N(\mathscr{H}_0) \, c_L(\mathscr{B}) \, \sum_{k=0}^Lz^{L-k}c_k(\mathscr{A})\right).
\end{equation}
Let $\mathscr{V}:=\mathscr{H}_0\oplus\mathscr{B}$ and $r:=N+L$.
Let $\mathscr{R}_j$ be the globally generated vector bundle on $Z_w$ given as the pullback of the twisted tangent bundle
$$
T_{\PP_\kk^{n-1}} \otimes \OO_{\PP_\kk^{n-1}}(-1) \;\cong\; {\rm Coker}\left(\OO_{\PP_\kk^{n-1}}(-1) \rightarrow \OO_{\PP_\kk^{n-1}}^n\right)
$$ 
on the $(n+1-j)$-th factor of $\PP=\left(\PP_\kk^{n-1}\right)^n$, so that $c_a(\mathscr{R}_j)=h_{n+1-j}^a$.
Consider the polynomial
$$
 \Phi(\bx,z,\tau)
 \;:=\;\int_{Z_w}c_\tau(\mathscr{V})c_z(\mathscr{A})
 \prod_{j=1}^n c_{x_j}(\mathscr{R}_j) \;\in\; \ZZ[\bx, z, \tau].
$$

Finally, by applying the Chern-class criterion of \cite[Proposition~6.2]{CRSegre} to the projection
$Z_w\to\PP^0$, we obtain that 
$$
\mathrm{N}(\Phi) \quad \text{ is a Lorentzian polynomial.}
$$ 
By \autoref{eq:ordinary-reciprocity}, the coefficient of
$\bx^\ba z^k$ in $F_w$ equals that of
$x_n^{n-1-a_1}\cdots x_1^{n-1-a_n}z^{L-k}$ in $P_w$.
Therefore \autoref{eq:ordinary-Chern-identity} and the fact that $A^*(\PP)\cong \ZZ[h_1,\ldots,h_n]/\left(h_1^n,\ldots,h_n^n\right)$ give
$$
 \mathrm{N}(F_w)
 \;=\;\partial_\tau^r\bigl(\mathrm{N}(\Phi)\bigr).
$$
Since derivatives preserve Lorentzianity~\cite{BH}, $\mathrm{N}(F_w)$ is Lorentzian, as claimed. 
\end{proof}

\begin{remark}\label{rem:ordinary-conjugation}
Here we verify the operator identity used in the proof of \autoref{prop:ordinary-support}.
We work in the Laurent polynomial ring
$\ZZ[x_1^{\pm1},\ldots,x_n^{\pm1},z^{\pm1}]$. 
Let
$$
 M \;:=\; (x_1\cdots x_n)^{n-1}
 \qquad \text{and} \qquad
 \mathcal{R}(f)(\bx,z) \;:=\; f(x_n^{-1},\ldots,x_1^{-1},z^{-1}).
$$
We have that $\iota(f)=M\mathcal{R}(f)$ and that $\mathcal{R}$ is a ring involution with $\mathcal{R}(M)=M^{-1}$. 
Since $s_j(M)=M$, we get that $D_j(Mh)=M D_j(h)$. 
It then follows
$$
 \iota\bigl(D_j(\iota(f))\bigr)
 \;=\; M\mathcal{R}\bigl(D_j(M\mathcal{R}(f))\bigr)
 \;=\; M\mathcal{R}(M)\mathcal{R}\bigl(D_j(\mathcal{R}(f))\bigr)
 \;=\; \mathcal{R}\bigl(D_j(\mathcal{R}(f))\bigr).
$$
Fix $1\leq i<n$ and set $j=n-i$. 
We have that 
$$
\mathcal{R}(x_j) \;=\; x_{i+1}^{-1}, \quad \mathcal{R}(x_{j+1}) \;=\; x_i^{-1} \quad \text{ and } \quad \mathcal{R}\bigl(s_j(\mathcal{R}(f))\bigr) \;=\; s_i(f).
$$
By combining everything, we obtain
$$
\begin{aligned}
 z(\iota D_{n-i}\iota)(f)
 &\;=\; z\frac{(x_i^{-1}+z^{-1})f
       -(x_{i+1}^{-1}+z^{-1})s_i(f)}
       {x_{i+1}^{-1}-x_i^{-1}}\\
 &\;=\; \frac{x_{i+1}(x_i+z)f-x_i(x_{i+1}+z)s_i(f)}
        {x_i-x_{i+1}}\\
 &\;=\; \partial_i\bigl(x_{i+1}(x_i+z)f\bigr) \;=\; \Theta_i(f).
\end{aligned}
$$
For a polynomial $f$, the resulting identity lies in $\ZZ[\bx,z]$.
\end{remark}

\section*{Acknowledgments}
The author received support from NSF grant DMS-2502321 and Simons Foundation Travel Support for Mathematicians Award MPS-TSM-00013551.

\medskip

\noindent
\textbf{AI disclosure.}
The author used OpenAI's GPT-6 for assistance with mathematical discussions and the revision of this manuscript. 
The author wrote the manuscript and takes full responsibility for the mathematical content and text.

\bibliographystyle{amsalpha}
\bibliography{double_grothendieck_snp}

@article{LS,
  author = {Lascoux, Alain and Sch{\"u}tzenberger, Marcel-Paul},
  title = {\href{https://igm.univ-mlv.fr/~berstel/Mps/Travaux/A/1982-2HopfCras.pdf}{{Structure de Hopf de l'anneau de cohomologie et de l'anneau de Grothendieck d'une vari{\'e}t{\'e} de drapeaux}}},
  journal = {C. R. Acad. Sci. Paris S{\'e}r. I Math.},
  volume = {295},
  year = {1982},
  pages = {629--633},
  url = {https://igm.univ-mlv.fr/~berstel/Mps/Travaux/A/1982-2HopfCras.pdf}
}

@inproceedings{FK,
  author = {Fomin, Sergey and Kirillov, Anatol N.},
  title = {\href{https://fpsac-archive.github.io/FPSAC94/ARTICLES/21Fomin.pdf}{{Grothendieck polynomials and the Yang--Baxter equation}}},
  booktitle = {{Proceedings of the 6th International Conference on Formal Power Series and Algebraic Combinatorics}},
  address = {New Brunswick, NJ},
  year = {1994},
  pages = {183--190},
  url = {https://fpsac-archive.github.io/FPSAC94/ARTICLES/21Fomin.pdf}
}

@article{Lenart,
  author = {Lenart, Cristian},
  title = {\href{https://doi.org/10.1007/PL00001276}{{Combinatorial aspects of the $K$-theory of Grassmannians}}},
  journal = {Ann. Comb.},
  volume = {4},
  year = {2000},
  number = {1},
  pages = {67--82},
  doi = {10.1007/PL00001276}
}

@article{LRS,
  author = {Lenart, Cristian and Robinson, Shawn and Sottile, Frank},
  title = {\href{https://doi.org/10.1353/ajm.2006.0034}{{Grothendieck polynomials via permutation patterns and chains in the Bruhat order}}},
  journal = {Amer. J. Math.},
  volume = {128},
  year = {2006},
  number = {4},
  pages = {805--848},
  doi = {10.1353/ajm.2006.0034},
  eprint = {math/0405539},
  url = {https://arxiv.org/abs/math/0405539}
}

@article{Buch,
  author = {Buch, Anders Skovsted},
  title = {\href{https://doi.org/10.1007/BF02392644}{{A Littlewood--Richardson rule for the $K$-theory of Grassmannians}}},
  journal = {Acta Math.},
  volume = {189},
  year = {2002},
  number = {1},
  pages = {37--78},
  doi = {10.1007/BF02392644},
  eprint = {math/0004137},
  url = {https://arxiv.org/abs/math/0004137}
}

@article{Weigandt,
  author = {Weigandt, Anna},
  title = {\href{https://doi.org/10.1016/j.jcta.2021.105470}{{Bumpless pipe dreams and alternating sign matrices}}},
  journal = {J. Combin. Theory Ser. A},
  volume = {182},
  year = {2021},
  pages = {105470},
  doi = {10.1016/j.jcta.2021.105470},
  eprint = {2003.07342},
  url = {https://arxiv.org/abs/2003.07342}
}

@book{AF,
  author = {Anderson, David and Fulton, William},
  title = {\href{https://doi.org/10.1017/9781009349994}{{Equivariant cohomology in algebraic geometry}}},
  series = {Cambridge Studies in Advanced Mathematics},
  volume = {210},
  publisher = {Cambridge University Press},
  address = {Cambridge},
  year = {2024},
  doi = {10.1017/9781009349994}
}

@article{CCLMZ,
  author = {Castillo, Federico and Cid-Ruiz, Yairon and Li, Binglin and Monta{\~n}o, Jonathan and Zhang, Naizhen},
  title = {\href{https://doi.org/10.1016/j.aim.2020.107382}{{When are multidegrees positive?}}},
  journal = {Adv. Math.},
  volume = {374},
  year = {2020},
  pages = {107382},
  doi = {10.1016/j.aim.2020.107382},
  eprint = {2005.07808},
  url = {https://arxiv.org/abs/2005.07808}
}

@article{BH,
  author = {Br{\"a}nd{\'e}n, Petter and Huh, June},
  title = {\href{https://doi.org/10.4007/annals.2020.192.3.4}{{Lorentzian polynomials}}},
  journal = {Ann. of Math. (2)},
  volume = {192},
  year = {2020},
  number = {3},
  pages = {821--891},
  doi = {10.4007/annals.2020.192.3.4},
  eprint = {1902.03719},
  url = {https://arxiv.org/abs/1902.03719}
}

@misc{CRSegre,
  author = {Cid-Ruiz, Yairon},
  title = {\href{https://arxiv.org/abs/2507.06424v2}{{Mixed Segre zeta functions and their log-concavity}}},
  year = {2026},
  note = {To appear in Algebra \& Number Theory},
  eprint = {2507.06424v2},
  url = {https://arxiv.org/abs/2507.06424v2}
}

@article{Brion,
  author = {Brion, Michel},
  title = {\href{https://doi.org/10.1016/S0021-8693(02)00505-7}{{Positivity in the Grothendieck group of complex flag varieties}}},
  journal = {J. Algebra},
  volume = {258},
  year = {2002},
  number = {1},
  pages = {137--159},
  doi = {10.1016/S0021-8693(02)00505-7},
  eprint = {math/0105254},
  url = {https://arxiv.org/abs/math/0105254}
}

@incollection{BrionFlags,
  author = {Brion, Michel},
  title = {\href{https://arxiv.org/abs/math/0410240}{{Lectures on the geometry of flag varieties}}},
  booktitle = {{Topics in cohomological studies of algebraic varieties}},
  series = {Trends in Mathematics},
  publisher = {Birkh{\"a}user},
  address = {Basel},
  year = {2005},
  pages = {33--85},
  eprint = {math/0410240},
  url = {https://arxiv.org/abs/math/0410240}
}

@article{CCMM,
  author = {Castillo, Federico and Cid-Ruiz, Yairon and Mohammadi, Fatemeh and Monta{\~n}o, Jonathan},
  title = {\href{https://doi.org/10.1017/fms.2023.101}{{Double Schubert polynomials do have saturated Newton polytopes}}},
  journal = {Forum Math. Sigma},
  volume = {11},
  year = {2023},
  pages = {e100},
  doi = {10.1017/fms.2023.101},
  eprint = {2109.10299},
  url = {https://arxiv.org/abs/2109.10299}
}

@article{CCC,
  author = {Caminata, Alessio and Cid-Ruiz, Yairon and Conca, Aldo},
  title = {\href{https://doi.org/10.1016/j.aim.2023.109361}{{Multidegrees, prime ideals, and non-standard gradings}}},
  journal = {Adv. Math.},
  volume = {435},
  year = {2023},
  pages = {109361},
  doi = {10.1016/j.aim.2023.109361},
  eprint = {2208.07238},
  url = {https://arxiv.org/abs/2208.07238}
}

@misc{CLM,
  author = {Cid-Ruiz, Yairon and Li, Yupeng and Matherne, Jacob P.},
  title = {\href{https://arxiv.org/abs/2411.17572}{{Log-concavity of polynomials arising from equivariant cohomology}}},
  year = {2024},
  eprint = {2411.17572},
  url = {https://arxiv.org/abs/2411.17572}
}

@misc{KMF,
  author = {Castillo, Federico and Cid-Ruiz, Yairon and Mohammadi, Fatemeh and Monta{\~n}o, Jonathan},
  title = {\href{https://arxiv.org/abs/2212.13091v3}{{$K$-polynomials of multiplicity-free varieties}}},
  year = {2025},
  eprint = {2212.13091v3},
  url = {https://arxiv.org/abs/2212.13091v3}
}

@article{Elkik,
  author = {Elkik, Ren{\'e}e},
  title = {\href{https://doi.org/10.1007/BF01578068}{{Singularit{\'e}s rationnelles et d{\'e}formations}}},
  journal = {Invent. Math.},
  volume = {47},
  year = {1978},
  number = {2},
  pages = {139--147},
  doi = {10.1007/BF01578068}
}

@book{EV,
  author = {Esnault, H{\'e}l{\`e}ne and Viehweg, Eckart},
  title = {\href{https://doi.org/10.1007/978-3-0348-8600-0}{{Lectures on vanishing theorems}}},
  series = {DMV Seminar},
  volume = {20},
  publisher = {Birkh{\"a}user},
  address = {Basel},
  year = {1992},
  doi = {10.1007/978-3-0348-8600-0}
}

@book{FOV,
  author = {Flenner, Hubert and O'Carroll, Liam and Vogel, Wolfgang},
  title = {\href{https://doi.org/10.1007/978-3-662-03817-8}{{Joins and intersections}}},
  series = {Springer Monographs in Mathematics},
  publisher = {Springer-Verlag},
  address = {Berlin},
  year = {1999},
  doi = {10.1007/978-3-662-03817-8}
}

@book{FultonIT,
  author = {Fulton, William},
  title = {\href{https://doi.org/10.1007/978-1-4612-1700-8}{{Intersection theory}}},
  series = {Ergebnisse der Mathematik und ihrer Grenzgebiete. 3. Folge},
  volume = {2},
  edition = {Second},
  publisher = {Springer-Verlag},
  address = {Berlin},
  year = {1998},
  doi = {10.1007/978-1-4612-1700-8}
}

@book{FP,
  author = {Fulton, William and Pragacz, Piotr},
  title = {\href{https://doi.org/10.1007/BFb0096380}{{Schubert varieties and degeneracy loci}}},
  series = {Lecture Notes in Mathematics},
  volume = {1689},
  publisher = {Springer-Verlag},
  address = {Berlin},
  year = {1998},
  doi = {10.1007/BFb0096380}
}

@article{HMMSD,
  author = {Huh, June and Matherne, Jacob P. and M{\'e}sz{\'a}ros, Karola and {St. Dizier}, Avery},
  title = {\href{https://doi.org/10.1090/tran/8606}{{Logarithmic concavity of Schur and related polynomials}}},
  journal = {Trans. Amer. Math. Soc.},
  volume = {375},
  year = {2022},
  number = {6},
  pages = {4411--4427},
  doi = {10.1090/tran/8606},
  eprint = {1906.09633v3},
  url = {https://arxiv.org/abs/1906.09633v3}
}

@article{KM,
  author = {Knutson, Allen and Miller, Ezra},
  title = {\href{https://doi.org/10.4007/annals.2005.161.1245}{{Gr{\"o}bner geometry of Schubert polynomials}}},
  journal = {Ann. of Math. (2)},
  volume = {161},
  year = {2005},
  number = {3},
  pages = {1245--1318},
  doi = {10.4007/annals.2005.161.1245}
}

@book{KollaMori,
  author = {Koll{\'a}r, J{\'a}nos and Mori, Shigefumi},
  title = {\href{https://doi.org/10.1017/CBO9780511662560}{{Birational geometry of algebraic varieties}}},
  series = {Cambridge Tracts in Mathematics},
  volume = {134},
  publisher = {Cambridge University Press},
  address = {Cambridge},
  year = {1998},
  doi = {10.1017/CBO9780511662560}
}

@article{Hudson,
  author = {Hudson, Thomas},
  title = {\href{https://doi.org/10.1017/is014005031jkt266}{{A Thom--Porteous formula for connective $K$-theory using algebraic cobordism}}},
  journal = {J. K-Theory},
  volume = {14},
  year = {2014},
  number = {2},
  pages = {343--369},
  doi = {10.1017/is014005031jkt266},
  url = {https://doi.org/10.1017/is014005031jkt266}
}

@article{Magyar,
  author = {Magyar, Peter},
  title = {\href{https://doi.org/10.1007/s000140050071}{{Schubert polynomials and Bott--Samelson varieties}}},
  journal = {Comment. Math. Helv.},
  volume = {73},
  year = {1998},
  number = {4},
  pages = {603--636},
  doi = {10.1007/s000140050071},
  url = {https://doi.org/10.1007/s000140050071}
}

@article{KW,
  author = {Koncki, Jakub and Weber, Andrzej},
  title = {\href{https://doi.org/10.1007/s00208-024-02953-2}{{Hecke algebra action on twisted motivic Chern classes and $K$-theoretic stable envelopes}}},
  journal = {Math. Ann.},
  volume = {391},
  year = {2025},
  pages = {1899--1964},
  doi = {10.1007/s00208-024-02953-2},
  eprint = {2301.12746},
  url = {https://arxiv.org/abs/2301.12746}
}

@article{Li,
  author = {Li, Shiyue},
  title = {\href{https://doi.org/10.1007/s44007-023-00054-1}{{Relative Bott--Samelson varieties}}},
  journal = {La Matematica},
  volume = {2},
  year = {2023},
  number = {2},
  pages = {420--437},
  doi = {10.1007/s44007-023-00054-1},
  eprint = {2011.04814},
  url = {https://arxiv.org/abs/2011.04814}
}

@article{MTY,
  author = {Monical, Cara and Tokcan, Neriman and Yong, Alexander},
  title = {\href{https://doi.org/10.1007/s00029-019-0513-8}{{Newton polytopes in algebraic combinatorics}}},
  journal = {Selecta Math. (N.S.)},
  volume = {25},
  year = {2019},
  number = {5},
  pages = {66},
  doi = {10.1007/s00029-019-0513-8},
  eprint = {1703.02583},
  url = {https://arxiv.org/abs/1703.02583}
}

@book{HH,
  author = {Herzog, J{\"u}rgen and Hibi, Takayuki},
  title = {\href{https://doi.org/10.1007/978-0-85729-106-6}{{Monomial ideals}}},
  series = {Graduate Texts in Mathematics},
  volume = {260},
  publisher = {Springer-Verlag},
  address = {London},
  year = {2011},
  doi = {10.1007/978-0-85729-106-6}
}

@book{Murota,
  author = {Murota, Kazuo},
  title = {\href{https://doi.org/10.1137/1.9780898718508}{{Discrete convex analysis}}},
  series = {SIAM Monographs on Discrete Mathematics and Applications},
  volume = {10},
  publisher = {SIAM},
  address = {Philadelphia},
  year = {2003},
  doi = {10.1137/1.9780898718508}
}

@book{Schrijver,
  author = {Schrijver, Alexander},
  title = {\href{https://link.springer.com/book/9783540443896}{{Combinatorial optimization. Polyhedra and efficiency. Vol.~B}}},
  series = {Algorithms and Combinatorics},
  volume = {24},
  publisher = {Springer-Verlag},
  address = {Berlin},
  year = {2003},
  note = {Matroids, trees, stable sets, Chapters 39--69},
  url = {https://link.springer.com/book/9783540443896}
}

@misc{NW,
  author = {Nguyen-Dang, Khai-Hoan and Wang, Zhenpeng},
  title = {\href{https://arxiv.org/abs/2609.02850}{{Chern flow and Chern moment algebras}}},
  year = {2026},
  eprint = {2609.02850v3},
  url = {https://arxiv.org/abs/2609.02850v3}
}

@article{PSW,
  author = {Pechenik, Oliver and Speyer, David E. and Weigandt, Anna},
  title = {\href{https://doi.org/10.1007/s00029-024-00959-x}{{Castelnuovo--Mumford regularity of matrix Schubert varieties}}},
  journal = {Selecta Math. (N.S.)},
  volume = {30},
  year = {2024},
  number = {4},
  pages = {66},
  doi = {10.1007/s00029-024-00959-x},
  eprint = {2111.10681v1},
  url = {https://arxiv.org/abs/2111.10681v1}
}

@misc{stacks-project,
  author = {The {Stacks Project Authors}},
  title = {\href{https://stacks.math.columbia.edu}{{The Stacks Project}}},
  year = {2026},
  url = {https://stacks.math.columbia.edu}
}

@article{CR_MIXED_MULT,
  author = {Cid-Ruiz, Yairon},
  title = {\href{https://doi.org/10.1016/j.jalgebra.2020.08.037}{{Mixed multiplicities and projective degrees of rational maps}}},
  journal = {J. Algebra},
  volume = {566},
  year = {2021},
  pages = {136--162},
  doi = {10.1016/j.jalgebra.2020.08.037}
}

@book{HARTSHORNE,
  author = {Hartshorne, Robin},
  title = {\href{https://doi.org/10.1007/978-1-4757-3849-0}{{Algebraic geometry}}},
  series = {Graduate Texts in Mathematics},
  volume = {52},
  publisher = {Springer-Verlag},
  address = {New York},
  year = {1977},
  doi = {10.1007/978-1-4757-3849-0}
}

@article{KT,
  author = {Kleiman, Steven and Thorup, Anders},
  title = {\href{https://doi.org/10.1006/jabr.1994.1182}{{A geometric theory of the Buchsbaum--Rim multiplicity}}},
  journal = {J. Algebra},
  volume = {167},
  year = {1994},
  number = {1},
  pages = {168--231},
  doi = {10.1006/jabr.1994.1182}
}

@article{FMS,
  author = {Fink, Alex and M{\'e}sz{\'a}ros, Karola and {St. Dizier}, Avery},
  title = {\href{https://doi.org/10.1016/j.aim.2018.05.028}{{Schubert polynomials as integer point transforms of generalized permutahedra}}},
  journal = {Adv. Math.},
  volume = {332},
  year = {2018},
  pages = {465--475},
  doi = {10.1016/j.aim.2018.05.028},
  eprint = {1706.04935},
  url = {https://arxiv.org/abs/1706.04935}
}

@article{EY,
  author = {Escobar, Laura and Yong, Alexander},
  title = {\href{https://doi.org/10.1016/j.crma.2017.07.003}{{Newton polytopes and symmetric Grothendieck polynomials}}},
  journal = {C. R. Math. Acad. Sci. Paris},
  volume = {355},
  year = {2017},
  number = {8},
  pages = {831--834},
  doi = {10.1016/j.crma.2017.07.003},
  eprint = {1705.07876},
  url = {https://arxiv.org/abs/1705.07876}
}

@article{MSD,
  author = {M{\'e}sz{\'a}ros, Karola and {St. Dizier}, Avery},
  title = {\href{https://doi.org/10.5802/alco.136}{{From generalized permutahedra to Grothendieck polynomials via flow polytopes}}},
  journal = {Algebr. Comb.},
  volume = {3},
  year = {2020},
  number = {5},
  pages = {1197--1229},
  doi = {10.5802/alco.136},
  eprint = {1705.02418},
  url = {https://arxiv.org/abs/1705.02418}
}

@article{HMSSD,
  author = {Hafner, Elena S. and M{\'e}sz{\'a}ros, Karola and Setiabrata, Linus and {St. Dizier}, Avery},
  title = {\href{https://doi.org/10.1137/23M1599082}{{$M$-convexity of vexillary Grothendieck polynomials via bubbling}}},
  journal = {SIAM J. Discrete Math.},
  volume = {38},
  year = {2024},
  number = {3},
  pages = {2194--2225},
  doi = {10.1137/23M1599082},
  eprint = {2306.08597},
  url = {https://arxiv.org/abs/2306.08597}
}

\end{document}